\documentclass[11pt]{amsart}
\usepackage[latin1]{inputenc}
\usepackage{graphicx}
\usepackage{amsmath}
\usepackage{amssymb}

\usepackage{amsmath,amsthm,amsfonts,amssymb,amscd}
\usepackage{enumerate}
\usepackage{amssymb,amsthm,amsmath,psfrag,mathrsfs}
\usepackage[latin1]{inputenc}
\usepackage{graphicx}
\usepackage{comment}

\def\leaderfill{\leaders\hbox to .8em{\hss .\hss}\hfill}
\def\_#1{{\lower 0.7ex\hbox{}}_{#1}}

\newtheorem{theorem}{Theorem}

\newtheorem{Theorem}{Theorem}[section]
\newtheorem{Corollary}[Theorem]{Corollary}

\newtheorem{Lemma}[Theorem]{Lemma}
\newtheorem{Claim}[Theorem]{Claim}
\theoremstyle{Definition}

\newtheorem{Assumption}[Theorem]{Assumption}
\theoremstyle{Remark}
\newtheorem{Remark}[Theorem]{Remark}

\def\leaderfill{\leaders\hbox to .8em{\hss .\hss}\hfill}
\def\_#1{{\lower 0.7ex\hbox{}}_{#1}}

\def\fa{{\mathcal{F}}}

\def\Re{\operatorname{{Re}}}

\def\mI{\operatorname{{Im}}}

\def\Diff{\operatorname{{Diff}}}
\def\sing{\operatorname{{sing}}}

\title[On a theorem of Lyapunov-Poincaré in higher dimension]{On a theorem of Lyapunov-Poincaré in higher dimension}
\author{V. Le\'on and B. Sc\'ardua}
\address{V. Le\'on. ILACVN - CICN, Universidade Federal da Integração Latino-Americana, Parque tecnológico de Itaipu, Foz do Iguaçu-PR, 85867-970 - Brazil}
\email{victor.leon@unila.edu.br}
\address{B. Sc\'ardua. Instituto de Matem\'atica - Universidade Federal do Rio de Janeiro,
CP. 68530-Rio de Janeiro-RJ, 21945-970 - Brazil}
\email{bruno.scardua@gmail.com}

\date{}
\begin{document}

\maketitle

\begin{abstract}The classical Lyapunov-Poincaré center theorem assures the existence of a first integral for an analytic one-form near a center singularity in dimension two, provided that the first jet of the one-form is nondegenerate. The basic point is the existence of an analytic first integral for the
given one-form. In this paper we consider generalizations for two main frameworks:
(i)  real analytic foliations of codimension one in higher dimension and (ii) singular holomorphic foliations in dimension two. All this is related to the problem of finding criteria assuring the existence of analytic first integrals for a given codimension one germ with a suitable first jet. Our approach consists in  giving an interpretation of the center theorem in terms of holomorphic foliations and, following an idea of Moussu, apply the holomorphic foliations arsenal in the obtaining the required first integral.
As a consequence we are able to revisit some of Reeb classical results on integrable perturbations of exact homogeneous one-forms, and prove some versions of these to the framework of non-isolated
(perturbations of transversely Morse type) singularities.
\end{abstract}


\section{Introduction and main results}
\label{section:introduction}
The classical Center theorem of Poincaré-Lyapunov (\cite{Lyapunov,Poincare}) reads as follows:
\begin{Theorem}
\label{Theorem:centertheorem} For a given germ of a real analytic one-form $\omega= a(x,y)dx + b(x,y)dy$
at the origin $0\in \mathbb R^2$, (i) having an isolated singularity for its first jet $\omega_1$ and
(ii) having a center at the origin (all leaves in a punctured neighborhood of the origin are diffeomorphic  to
the circle), then $\omega$ admits a real analytic first integral in the strong sense, which means that
$\omega=gdf$ for some real analytic function germs $f,g$ at $0\in \mathbb R^2$, with $g(0)\ne 0$ and
$f$ having a Morse singular point at the origin.
\end{Theorem}
There are some equivalent statements also in terms of vector fields.
Besides the classical analytical proofs, there is a quite geometrical proof given by Moussu
\cite{moussu}. In his paper he makes use of the complexification of the one-form,
obtaining therefore a holomorphic one-form with a suitable singularity at the origin $0\in \mathbb C^2$. Moussu's approach strongly relies on the Mattei-Moussu theorem (\cite{mattei-moussu}), about topological (dynamical) conditions assuring the existence of
holomorphic first integrals for germs of holomorphic foliations near a singular point (Theorem~B page 473). The center condition then together with Mattei-Moussu theorem above mentioned, assures the existence of a first integral for  the complexification and therefore for the
real analytic one-form. Moussu's ideas are quite attractive and inspiring.
They also show the interplay between real analytic dynamical systems and
geometric theory of holomorphic foliations.

In this paper we address problems motivated by the above statement.
Our first result in this direction reads as follows:

\begin{Theorem}
\label{Theorem:newcenter}
For a given germ of a real analytic one-form $\omega= a(x,y)dx + b(x,y)dy$
at the origin $0\in \mathbb R^2$,  having an isolated singularity for its first jet $\omega_1$, the following conditions are equivalent for the induced foliation germ $\fa: \, \omega=0$:

\begin{enumerate}[(i)]
\item The leaves of  $\fa$ are closed in some small neighborhood of the origin.
\item The origin is a center singularity.
\item There is a real analytic first integral.
\end{enumerate}
\end{Theorem}

Clearly, in view of Theorem~\ref{Theorem:centertheorem}, the main point is $(i) \implies (ii)$.
Theorem~\ref{Theorem:newcenter} above may look like a too small improvement in the classical statement of Lyapunov-Poincaré. Nevetheless, its applications prove its usefulness. 

\vglue.1in
In the course of the proof of Theorem~\ref{Theorem:newcenter} we shall obtain (cf. Lemma~\ref{Lemma:orbitsperiodic}):
\begin{Corollary}
\label{Corollary:vfcenter}

For a germ  at $0 \in \mathbb R^2$ of a  real analytic vector field $X$ having  first jet  of the  form $X_1 = x_2 \frac{\partial}{\partial x_1} - x_1 \frac{\partial}{\partial x_2}$,
 the following conditions are equivalent:

\begin{enumerate}[{(i)}]

\item The orbits of $X$ are closed subsets in some neighborhood of the origin.
\item $X$ has a center type singularity at the origin.
\item $X$ admits a real analytic first integral.
\item $X$ is analytically almost linearizable, i.e., $X$ is a multiple of a linear vector field after an analytic local change of coordinates.
\end{enumerate}
\end{Corollary}

 Our next result deals with  higher dimension versions of Theorem~\ref{Theorem:centertheorem}.

\begin{theorem}
\label{Theorem:reebtype}
Let $\fa$ be a real analytic codimension one singular foliation given in a neighborhood $U\subset \mathbb R^n$ of the origin
$0 \in \mathbb R^n, \,  n \geq 3$ by an integrable one-form $\omega$ having first jet of the form
$\omega_1= d(\sum\limits_{j=1}^r x_j ^2 ), \, 2\leq r \leq n$.
Then $\fa$ admits an analytic first integral in some neighborhood of the origin in the following situations:
\begin{enumerate}[{(i)}]
\item If $r=2$ and the leaves of $\fa$ are closed  in some neighborhood of the origin.

\item If $ 3 \leq r \leq n$.

\end{enumerate}
\end{theorem}

\begin{Remark}
{\rm
In both cases we have:

\begin{enumerate}[{\rm(a)}]
\item $\fa$ admits an analytic linearization, i.e.,  $\fa$ is given by $d(\sum\limits_{j=1}^r (\tilde x_j)^2)=0$ in suitable analytic coordinates $(\tilde x_1,\ldots,\tilde x_n)$.
\item The leaves of $\fa$ are closed diffeomorphic to the cylinder $S^{r-1}\times \mathbb R^{n-r}$ in some neighborhood of the origin.

\end{enumerate}

}\end{Remark}

We observe that Theorem~\ref{Theorem:reebtype} can be seen as a version of a classical theorem of
Reeb, for the case of degenerate quadratic forms in the first jet of the integrable one-form.
Let us spend a few words about it.
First we recall the above mentioned classical result due to
G. Reeb \cite{Reeb} (see also \cite{C-LN} page 85):
\begin{Theorem}[Reeb, \cite{Reeb}]
\label{Theorem:Reeblinearization} Let $\omega$ be an analytic integrable 1-form
defined in a neighborhood of the origin $ 0 \in \mathbb R^n, n \geq
3$. Suppose that $\omega(0)=0$ and $\omega$ has a non-degenerate
linear part $\omega_1 = df$, i.e., $f$ is a quadratic form of
maximal rank (not necessarily of center type). Then there exist an
analytic diffeomorphism $h \colon (\mathbb R^n,0) \to (\mathbb
R^n,0)$ and an analytic function $g \colon (\mathbb R^n,0) \to
(\mathbb R,0)$ with $h^*(\omega)=gdf$.
\end{Theorem}

The above theorem has some version for $\omega$ of class $C^2$ but
demanding that the singularity is of center type (see \cite{C-LN} page 84 or
\cite{Reeb}).
We point-out that part (ii) in our Theorem~\ref{Theorem:reebtype} extends Reeb's theorem (Theorem~\ref{Theorem:Reeblinearization})  above to the case of degenerate center singularities.

In \cite{Cerveau-Scardua2} the authors  consider some versions of Reeb's theorem above.
They work with holomorphic integrable 1-forms of type $\Omega=
dP + \Omega^\prime$ where $P$ is a homogeneous irreducible
polynomial, and $\Omega^\prime$ is a 1-form of higher order terms
than $dP$. Under some regularity hypotheses on $P$ they also conclude that
$\Omega$ admits a first integral which is a perturbation of $P$.
This includes for instance the case $P=\sum\limits_{j=1}^n x_j ^d,
\, n \geq 3, d \geq 2$ a so called {\it Pham polynomial}. Given a
polynomial $P \in \mathbb R[x_1,\ldots,x_n]$ we denote by $P^{\mathbb
C} \in \mathbb C[z_1,\ldots,z_n]$ its complexification where $z_j= x_j
+ \sqrt{-1} y_j$.
The above is the main motivation for our  next result that reads as follows:

\begin{theorem}
\label{Theorem:reebtypepham}
Let $\fa$ be a real analytic codimension one singular foliation given in a neighborhood $U\subset \mathbb R^n$ of the origin
$0 \in \mathbb R^n, \,  n \geq 3$ by an integrable one-form $\omega= dP_{r,n,d}  + P_{r,n,d} \, {\tilde \omega}$ where $P_{r,n,d}$ is the truncated Pham polynomial $P_{r,n,d}=\sum\limits_{j=1}^r x_j ^d, \,   3 \leq r \leq n $ and ${\tilde \omega}$ is an analytic one-form. If  $d=p^s$ for some prime number $p\in \mathbb N$
then $\fa$ admits an analytic first integral in some neighborhood of the origin.

\end{theorem}

Notice that $P_{r,n,d}$ in the case $d\geq 2, \, r<n$  has a non-isolated singularity at the origin.
We observe that Theorem~\ref{Theorem:reebtypepham} does not hold for $r=2$ as can be seen from the following example. Take $\omega(x,y,z)=d(x^4 + y^4) - 2x^2 y^2 dy$ in $\mathbb C^3$. Then $\omega$ is clearly integrable (it depends only on two variables), it has a center type singularity at the origin and also its first nonzero jet is the differential of a truncated Pham polynomial $P_{2,3,2}$. Nevertheless, working with power series, it can be easily shown that $\omega$ does not admit a real analytic first integral.

\subsection{Complex analytic foliations}
In what follows, by {\it germ of a holomorphic foliation at the origin $0\in \mathbb C^2$} we shall mean
a germ of a holomorphic foliation by curves, with an isolated singularity at the origin $0 \in \mathbb C^2$. As already mentioned, our approach for proving Theorem~\ref{Theorem:reebtype} follows the idea of complexification of the problem, as suggested by \cite{moussu}. Indeed, it is based in the following variant of Mattei-Moussu theorem:

\begin{Theorem}
\label{Theorem:totreal}

Let $\fa$ be a germ of a holomorphic foliation at the origin $0 \in \mathbb C^2$,
 given by $\omega=0$ where  $\omega=d(xy) + {\tilde \omega}$ and ${\tilde \omega}$ has jet of order one equal to zero.
Then the following conditions are equivalent:

\begin{enumerate}[{(i)}]

\item $\fa$ admits a holomorphic first integral of the form $fg$ for irreducible germs $f, g \in \mathcal O_2$ in general position.

\item There is   a germ of an analytic dimension two variety $V^2\subset \mathbb C^2$ with $0 \in V^2$, having contact order one with  $\fa$ outside of the origin and
such that the restriction of $\fa$ has closed leaves in $V^2$.

\end{enumerate}

\end{Theorem}

In the situation of the above theorem we also have:

\begin{itemize}

\item {\it There is   a germ of a totally real analytic variety $V^2\subset \mathbb C^2$ having contact order one with $\fa$ and
such that the restriction of $\fa$ to $V^2$  has a center type singularity at the origin  in $V^2$.}

\end{itemize}

The notions of generic position, order one contact and totally real will be recalled in \S~\ref{section:holomorphic}.

Our Theorem~\ref{Theorem:totreal} above has connections with the main result in \cite{Cerveau-LinsNeto} where the authors prove the existence of a meromorphic first integral for a of a codimension one holomorphic foliation at $0\in \mathbb C^n, n \geq 2$  provided that it is  tangent to a germ at $0\in \mathbb C^n$  of a real codimension one and irreducible analytic variety $M$.

\section{Holomorphic foliations: proof of Theorem~\ref{Theorem:totreal}}
\label{section:holomorphic}
A few words about the notions in the statement of Theorem~\ref{Theorem:totreal}.
Two irreducible and reduced germs $f, g \in \mathcal O_2$ with $f(0)=g(0)=0$ are {\it in general position} if $(f=0)$ and $(g=0)$ meet transversely at the origin.
We also recall that a  submanifold $V$ of a Kaehler complex manifold $M$ is called
{\it  totally real}  if the complex structure $J\colon TM \to TM$ of $M$ maps
each tangent space $T_p V\subset T_pM $ of $V$  into  the normal space $(T_p V)^\perp\subset T_p M$.
We refer to \cite{Chern} for a detailed exposition, examples and  characterizations of totally real manifolds.
As of now we mention that:  (i)  given a holomorphic function $ F \colon  M \to \mathbb C$, the smooth part of the
subvarieties  $(\Re F=0)\subset M$ and $(\mI F=0)$ are totally real submanifolds. (ii) the intersection of  two
totally real submanifolds is totally real at the points where this intersection is transverse.

In particular, given two germs of holomorphic functions $f, g \colon \mathbb C^2 \to \mathbb C$, in general position,
and vanishing at $0\in \mathbb C^2$, then the intersection
$V^2 = (\Re(f)= \Re(g))\cap (\mI(f)=- \mI(g))$ is a germ of a  totally real surface at the origin $0 \in \mathbb C^2$.

In Theorem~\ref{Theorem:totreal} above the leaves of $\fa$ are of real dimension two, in a space of real dimension four. Thus, condition (ii) is equivalent to the following:

\begin{itemize}
\item[{\rm (ii)'}] There is   a germ of a totally real analytic surface $V^2\subset \mathbb C^2$ with $0 \in V^2$ and
such that the restriction of $\fa$ has closed leaves in $V^2$.
\end{itemize}

Given a real foliation $\fa$ of codimension $k$ in a differentiable manifold $M$ and an immersed connected submanifold $V\subset M$, the {\it contact order} of $\fa$ with $V$ at a point $p\in V$ is the dimension of the intersection $T_p(V)\cap T_p(\fa) \subset T_p(M)$ as linear subspaces of the tangent space $T_p(M)$. We say that $\fa$ has contact order $r$ with $V$ if their contact order is $r$ at each point $p\in V$. In the case where $\fa$ is a holomorphic foliation of (complex) codimension one in an open subset $U\subset \mathbb C^2$ with $\sing(\fa)=\{0\}\subset U$, and $V^2\subset U$ is a
real surface, we have:
\begin{itemize}
\item $V^2$ is transverse to $\fa$ off the origin iff $V^2\setminus\{0\}$ and $\fa$ have contact order equal to zero.

\item $V^2$ is $\fa$ invariant iff $V^2\setminus \{0\}$ and $\fa$ have contact order equal to $2$.

\item $V^2\setminus \{0\}$ has contact order with $\fa$ equal to $1$ iff $V^2\setminus \{0\}$ is a totally real submanifold not invariant by $\fa$.

\end{itemize}

Let us now prove Theorem~\ref{Theorem:totreal}.

\begin{proof}[Proof of Theorem~\ref{Theorem:totreal}]

First we assume that $\fa$ admits a holomorphic first integral of the form $fg$ with $f, g \in \mathcal O_2$, $f(0)=g(0)=0$,
$f$ and $g$ (being germs reduced and  irreducible and)
in general position. We consider the analytic varieties of real codimension one $\mathcal  R: (\Re f = \Re g)\subset \mathbb R^4$ and $\mathcal I : (\mI f= - \mI g)\subset \mathbb R^4$. Since $f$ and $g$ are in general position the intersection
$\mathcal R \cap \mathcal I =V^2$ is a two-dimensional analytic variety. Also $0 \in V^2$ because $f$ and $g$ vanish at the origin. Let us now put $X=\frac{f+ g}{2}$ and $Y=\frac{f - g }{2 i}$. Then $f= X + iY$ and $g = X - i Y$ and therefore
$fg=X^2 + Y^2$. Moreover, in the variety $V^2$ we have $X=\Re(f)=\Re(g)$ and $Y=\mI(f)=-\mI(g)$ so that, restricted to $V^2$ we have
$fg=||f||^2=||g||^2$. This shows that the restriction to $V^2$ of the foliation  $\fa$ is a real analytic foliation by curves which are closed. In particular, the contact order of $\fa$ with $V^2$ is one. Indeed the restriction $\fa\big|_{V^2}$  gives an analytic center type singularity at the origin $0\in V^2$. Finally, since $\fa$ is holomorphic and has contact order equal to one with $V^2$ it follows that  $V^2$ is a totally real subvariety.
This proves $(i) \implies (ii)$ in Theorem~\ref{Theorem:totreal}.

Let us now prove $(ii) \implies (i)$. From hypothesis (ii) and from the considerations right after Theorem~\ref{Theorem:totreal} we conclude that:
{\em There is   a germ of a totally real analytic variety $V^2\subset \mathbb C^2$ having contact order one with $\fa$ and
such that the restriction of $\fa$ to $V^2$  has a center type singularity at the origin  in $V^2$.}
Up to an analytic change of coordinates in $\mathbb C^2$ we may assume that $V^2\subset \mathbb C^2$ corresponds to the totally real space $\mathbb R^2 \subset \mathbb C^2$, i.e., in suitable local coordinates $(x,y)\in \mathbb C^2$ we have $V^2: (\mI (x) = \mI (y)=0)$.
Assume now that we have a holomorphic foliation $\fa$ defined in a neighborhood of the origin $0\in \mathbb C^2$ by a
one-form $\omega=d(xy) + {\tilde \omega}$ where ${\tilde \omega}$ has zero jet of order one at the origin.
We know that $\fa: \omega=0$ corresponds to a Siegel singularity at the origin, since it is given by a one-form
with linear part $\omega_1:=xdy + ydx$. The blow-up $y=tx$ at the origin produces a foliation of the form
$2tx dx + x^2dt + {\tilde \omega}(x,tx)=0$. Thus we have a singularity of Siegel type on the origin of the system $(x,t)$ given by
$2t dx + xdt +\ldots=0$. In the coordinate system $x=uy$ we have $2uy dy + y^2 du + {\tilde \omega}(uy,y)=0$ and then we have a singularity of Siegel type at the origin of this system given by $2u dy + ydu + \ldots=0$.

Now we make an assumption:

\begin{Assumption}
Assume that $\fa$ is the complexification of a real analytic foliation $\fa_{\mathbb R}$ which has
a center type singularity at the origin $0\in \mathbb R^2$.

\end{Assumption}

The above assumption means that $\fa$ has contact order one with  the real space $\mathbb R^2 \subset \mathbb C^2$ and its restriction to this space exhibits a center type singularity at the origin $0\in \mathbb R^2$.
Recall that the real space above is given by $\mI(x)=\mI(y)=0$ where $(x,y)\in \mathbb C^2$ are affine
coordinates in $\mathbb C^2$.

The inverse image of this real plane in the blow-up $\tilde{\mathbb C^2 _0}$ corresponds to a Moebius band $M^2$
through the equator of the exceptional divisor $\mathbb E\simeq \mathbb CP(1)$. The pull-back foliation $\fa^*$
in $\tilde{\mathbb C^2 _0}$ then leaves invariant this Moebius band and has only closed leaves in $M^2$. Indeed, since $\fa_\mathbb R$ has a center type singularity at the origin, the foliation $\fa^*$ restricted to $M^2$ has closed compact leaves in a neighborhood of the equator in $M^2$. Now we consider the projective holonomy group of the exceptional divisor $E$. This means the holonomy group of the leaf $E\setminus \sing(\fa^*)$ for the foliation $\fa^*$. From what we have seen above, this foliation has exactly two singularities in $E$, corresponding to the north and south poles of $E$. Thus the holonomy group above mentioned is generated by a simple loop around the equator, i.e, this is a cyclic group.
Let us denote by $h$ a generator of this group obtained as follows. Choose a point $p\in E$ and a local transverse disc $\Sigma$ to $E$ centered at $p$. Then denote by $h\colon (\Sigma,p) \to (\Sigma,p)$ the holonomy map corresponding to
the equator $\gamma=M^2 \cap E$. Notice that, since $E$ is invariant by $\fa^*$, the equator $\gamma$ corresponds to a  compact leaf (periodic orbit) of the induced foliation in $M_2$. Because the leaves of $\fa^*$ in $M^2$ are all compact in a neighborhood of $\gamma$, this implies that the holonomy map (Poincaré map) corresponding to $\gamma$ regarding $\fa^*\big|_{M^2}$ is a periodic map of order two. Thus the $\fa^*$-holonomy map $h$ admits a real analytic curve $\gamma \cap \Sigma$ where its orbits are periodic of period $\leq 2$. Since $\gamma$ contains the origin,  $h$ is a periodic map
of period $2$. This implies, by standard methods described in \cite{mattei-moussu} that the foliation $\fa$ admits a holomorphic first integral. Now we claim that this first integral is of the form $fg$ where $f, g \in \mathcal O_2$ are irreducible and reduced and, up to reordering $f$ and $g$, we must have  $x\big|f$ and $y\big| g$ in $\mathcal O_2$. This is not difficult to see since $\fa$ has a Siegel type singularity at the origin, of the form $d(xy) + {\tilde \omega}=0$ and this implies that there are exactly two (transverse) separatrices  through the singular point at the origin. These separatrices are given by given $(xy=0)$. Since $(f=0)$ and $(g=0)$ correspond to separatrices of $\fa$ the result follows.

\end{proof}

\section{Proof of Theorem~\ref{Theorem:newcenter}}

Let us first state a few of lemmas we shall need.

\begin{Lemma}
\label{Lemma:finite}
Let $h\in\Diff(\mathbb C,0)$ be a germ of holomorphic diffeomorphism
tangent to the identity, i.e., $h(z) = z + a_{k+1} z^{k+1} + \ldots$
Assume that there is a real analytic invariant curve $\gamma$ through the origin
$0\in \mathbb C$ such that the pseudo-orbits of $h$ in $\gamma$ are closed. Then
$h$ is the identity.
\end{Lemma}

\begin{proof}[Proof of the lemma]
We shall use the well-known topological description of the
germs tangent to the identity in dimension one due to Camacho  (\cite{camacho}) and Leau (\cite{Leau}).
From this description, if the map is not the identity the only invariant curves through the origin
where the orbits are closed are the trivial ones, i.e, the origin itself.
\end{proof}

Let $\fa$ be a real analytic codimension one foliation with singularities in a neighborhood of the origin $0 \in \mathbb R^n$. This means that $\fa$ is defined by a real analytic one-form
$\omega=\sum\limits_{j=1}^n a_j(x)dx_j$, defined in a neighborood of the origin, and satisfying the integrability condition $\omega \wedge d\omega=0$. We consider the complexification of $\fa$ which we will denote by $\fa_{\mathbb C}$. This is a codimension one holomorphic foliation with singularity, defined in a neighborhood of the origin $0\in \mathbb C^n$ by the complexification $\omega_{\mathbb C}$ of the form $\omega$.
In complex coordinates $(z_1,\ldots,z_n)$ we can write $z_j=x_j + i y_j$ and $\omega_{\mathbb C}=
d(\sum\limits_{j=1}^n z_j ^2) + {\tilde \omega}_{\mathbb C}$ for some one-form ${\tilde \omega}_{\mathbb C}$ with zero first jet at the origin. Now we consider the real space $\mathbb R^n \subset \mathbb C^n$ given by $y_j=0, j=1,\ldots,n$.

The next result is a well-known easy to prove lemma:
\begin{Lemma}
\label{Lemma:complexificationfirstintegral}
Let $\fa$ be a real analytic foliation in a neighborhood of the origin $0 \in \mathbb R^n$ whose complexification $\fa_{\mathbb C}$ admits a
holomorphic first integral. Then $\fa$ admits a real analytic first integral, defined in some neighborhood of the origin.
Indeed, there is a real analytic first integral $f$ for $\fa$ such that the complexification $f_{\mathbb C}$ of $f$ is a
holomorphic first integral for $\fa_{\mathbb C}$.
\end{Lemma}

The main point is the following:

\begin{Lemma}
\label{Lemma:orbitsperiodic}
Let $X$ be a real analytic vector field in a neighborhood of the origin $0\in \mathbb R^2$, having an isolated singularity at the origin and linear part at this singularity given by $DX(0)=x_1 \frac{\partial}{\partial x_2} - x_2 \frac{\partial}{\partial x_1}$. Assume also that the orbits of $X$ are closed (in the classical sense of topology) in some neighborhood of the origin. Then these orbits are periodic in some neighborhood of the origin and the origin is a center type singularity for $X$.

\end{Lemma}

\begin{proof}
Indeed, the complexification $X_{\mathbb C}$ of $X$ is a complex analytic vector field defined in a neighborhood of the origin $0\in \mathbb C^2$. In complex affine coordinates $(x,y)\in \mathbb C^2$ we have $X_{\mathbb C}= x \frac{\partial}{\partial x} -y \frac{\partial}{\partial y} + X_2$ where $X_2$ has a zero order one jet at the origin.
Then $X_{\mathbb C}$ generates a holomorphic foliation $\fa_{\mathbb C}$ with an isolated Siegel type singularity at the origin, of the form $xdx + ydy + \ldots=0$. Then $\fa_{\mathbb C}$ is in the Siegel domain and we may assume that the coordinate axes are invariant (\cite{mattei-moussu}). In this case the quadratic  blow-up of $\mathbb C^2$ at the origin induces a foliation
$(\fa_{\mathbb C})^*$ in the blow-up space $\widetilde{\mathbb C^2 _0}$ which leaves invariant the exceptional divisor $E\simeq \mathbb CP(1)$ and has exactly two singularities, the north and south poles, in $E$, both of Siegel type.
The equator $\gamma$ then generates the projective holonomy of $E$ relatively to $(\fa)^*$ via a germ of a holomorphic diffeomorphism $h(z)= e^{i\pi} z + \ldots$. This map $h$ once evaluated in a suitable transverse disc $\Sigma\simeq \mathbb D$
centered at some point $p \in \gamma$ and transverse to $E$, leaves invariant a real analytic curve $\Gamma\subset \Sigma$, corresponding to the intersection of the inverse image of the real plane $\mathbb R^2 : (\mI (x) = \mI (y)=0)$ with the transverse section $\Sigma$. Restricted to $\Gamma$ the pseudo-orbits of $h$ are closed. This does not mean that the trajectories of $X$ are periodic.  Now applying Lemma~\ref{Lemma:finite} we conclude that $h$ is periodic of period two.
From Mattei-Moussu theorem (\cite{mattei-moussu} page 473) the foliation $\fa_\mathbb C$ admits then a holomorphic first integral in a neighborhood of the origin $0\in \mathbb C^2$. From Lemma~\ref{Lemma:complexificationfirstintegral} we conclude that the vector field $X$ admits an analytic first integral (not necessarily a potential function). Let us denote by $f\colon U,0 \to \mathbb R,0$ an analytic first integral of $X$. This means that $X(f)=0$, i.e., $f$ is constant on each orbit of $X$ in $V$.
Thanks to the linear part of $X$ we may assume that $f(x_1,x_2) = x_1 ^2 + x_2 ^2 + hot$ and the thanks to Morse lemma
we conclude that the origin is a center singularity for $X$.

\end{proof}

\begin{proof}[Proof of Theorem~\ref{Theorem:newcenter}]
As mentioned in the introduction,  the main point is $(i) \implies (ii)$. Let us then assume that the leaves are closed in some small neighborhood of the origin. According to Lemma~\ref{Lemma:orbitsperiodic} the origin is a center singularity. Evoking then Lyapunov-Poincaré theorem (Theorem~\ref{Theorem:centertheorem}) we conclude that $\fa$ admits a real analytic first integral.
\end{proof}

\begin{proof}[Proof of Corollary~\ref{Corollary:vfcenter}]
Lemma~\ref{Lemma:orbitsperiodic} shows that $(i)\implies (ii)$. Theorem~\ref{Theorem:centertheorem}
shows that $(ii)\implies (iii)$. Classical Morse lemma shows that $(iii)\implies (iv)$.
Finally, $(iv)\implies (i)$ is straightforward from the fact that the linear  part of
$X$ admits the first integral $x_1^2 + x_2 ^2$.

\end{proof}
\section{Proof of Theorem~\ref{Theorem:reebtype}}

Now we are in  the following situation. We have a codimension one real analytic singular foliation $\fa$ defined in a neighborhood $U$ of the origin
$0\in \mathbb R^n, n \geq 3$. We assume that $\fa$ is of the form $\omega=0$ where $\omega$ is integrable real analytic and
writes as $\omega=d(\sum\limits_{j=1}^r x_j ^2) + {\tilde \omega}$ where the first jet of ${\tilde \omega}$ at the origin is zero. A very first remark is that we may suppose $ r \leq n-1$. Indeed, the case $r=n$
is covered by Reeb's theorem (Theorem~\ref{Theorem:Reeblinearization}).

Let us prove (i). For this sake we shall make the following assumption:
\begin{Assumption}
We have $r=2$ and the leaves of $\fa$ are closed  in some neighborhood of the origin.
\end{Assumption}

We consider the complexification of $\fa$ which we will denote by $\fa_{\mathbb C}$.

\noindent{\bf First case $r=2, \, n =3$}.
 In this case the hypersurfaces given by $d(\sum\limits_{j=1}^r x_j ^2)=0$ are coaxial cylinders
with axis on the $x_3$-axis. Let us denote by $\mathbb R^2 \cong E^2\subset \mathbb R^3$ a real plane given by $x_3=A x_1 + B x_2$ for some coefficients $A, B$ such that $E$ is in general position with respect to $\fa$. For simplicity we shall write $E=\mathbb R^2$. The restriction $\fa\big|_{\mathbb R^2}$ is then a foliation with an isolated singularity at the origin and given by a one-form
$d(x_ 1^2 + x_ 2 ^2) + \ldots=0$. Moreover, by hypothesis the leaves are closed so that by Theorem~\ref{Theorem:newcenter} we know that $\fa\big|_{\mathbb R^2}$ admits a real analytic first integral, indeed, it is analytically linearizable.
Let us denote by $h\colon (\mathbb R^2, 0) \to (\mathbb R, 0)$
a quadratic first integral for $\fa\big|_{\mathbb R^2}$ defined in a neighborhood of the origin $0 \in \mathbb R^2$.
Then the complexification $h_{\mathbb C}$ of $h$ is a holomorphic first integral  for the complexification of $\fa\big|_{\mathbb R^2}$ to $\mathbb C^2$. Since the operators "restriction" and "complexification" commute,
we know that the restriction of the complexification $\fa_{\mathbb C}$ to $\mathbb C^2$ is the complexification of the restriction $\fa\big|_{\mathbb R^2}$. Thus we have shown that $\fa\big|_{\mathbb C^2}$ admits a holomorphic first integral.

For a suitable choice of the plane $E: x_3 = A x_1 + B x_2$ we may assume that:

\begin{Assumption}
Assume that the complex plane $\mathbb C^2 \subset \mathbb C^3$ obtained from $E$  is in general position with
respect to $\fa_{\mathbb C}$.

\end{Assumption}
Just to preserve the simplicity of ideas, if $E=\mathbb R^2$ is given by $x_3=0$ then
$\mathbb C^2\subset \mathbb C^3$ above mentioned will be given by $z_3=0$.

Under the above assumption, according to \cite{mattei-moussu} the existence of a holomorphic first integral in $\mathbb C^2$ assures the existence of a holomorphic first integral for $\fa_{\mathbb C}$ in $\mathbb C^3$. This completes this part.

Now we consider the remaining case for $r=2$, i.e., $n \geq 3$. Let us for instance assume that $n=4$.
Given a generic linearly embedded hyperplane $\mathbb R^3 \cong E^3\subset \mathbb R^4$, given by
some equation $x_4 = A x_1 + B x_2 + C x_3$ for generic coefficients $A, B, C$ we may consider the restriction $\fa\big|_{E}$. This foliation in $\mathbb R^3$ is subject to the already considered case
$r=2, n=3$. Thus we may conclude that $\fa\big|_{E}$ admits an analytic first integral defined in some neighborhood of the origin $0 \in E\cong \mathbb R^3$. By Lemma~\ref{Lemma:complexificationfirstintegral} the complexification $(\fa\big|_{E})_{\mathbb C}$ of
this foliation, is a foliation in neighborhood of the origin $0\in \mathbb C^3\subset \mathbb C^4$, and this foliation germ admits a holomorphic first integral.
Let us denote, as usual, by $\fa_{\mathbb C}$ the complexification of $\fa$.
Moreover, as already observed, we have $\fa_{\mathbb C}\big|_{\mathbb C^3}= (\fa\big|_{\mathbb R^3})_{\mathbb C}$, i.e,
$\fa_{\mathbb C}$ is the extension to $\mathbb C^4$ of the complexification of the restriction of $\fa$ to $\mathbb R^3$. In particular, the restriction of $\fa_{\mathbb C}$ to $\mathbb C^3$ admits a holomorphic first integral. The plane $\mathbb C^3$ may be assumed to be in general position with respect to $\fa$ in $\mathbb C^4$. Hence, according to \cite{mattei-moussu}, the existence of a holomorphic first integral for $(\fa_{\mathbb C})\big|_{\mathbb C^3}$ implies the existence of a
holomorphic first integral for $\fa_{\mathbb C}$ in some neighborhood of the origin $0 \in \mathbb C^4$. By Lemma~\ref{Lemma:complexificationfirstintegral} the foliation $\fa$ admits a real analytic first integral in some neighborhood of the origin $0\in \mathbb R^4$.
The case $n \geq 5$ follows from this type argument in an induction process.
This ends the proof of (i).

Let us now prove (ii). In order to fix the ideas, we will first consider the:

\noindent{\bf  case $3 \leq r=n-1$}
Let us start with the case $r=3$ and $n=4$.
The corresponding linear foliation has leaves diffeomorphic to the cylinder $S^2\times \mathbb R$ in $\mathbb R^4$. Moreover, the original foliation  is given by a one-form $\omega=d(x_1^2 + x_2 ^2 + x_3 ^2) + {\tilde \omega}(x_1,\ldots,x_4)$. The procedure is pretty similar to the one adopted for the case $r=2$ and $n=3$. Indeed, we consider a hyperplane $\mathbb R^3 \cong E\subset \mathbb R^4$ in general position with respect to $\fa$, given by $x_4=a_1 x_1 + a_2 x_2 + a_3 x_3$ for some suitable choice of $a_1,a_2,a_3$. The restriction $\fa\big|_{E}$ is then a foliation
given by a one-form $\omega=d(x_1^2 + x_2 ^2 + x_3 ^2) + {\tilde \omega}(x_1,\ldots,x_4)$.
Then Reeb's theorem (Theorem~\ref{Theorem:Reeblinearization}) implies that $\fa\big|_{E}$ admits an analytic first integral in some neighborhood of the origin $0 \in E \cong \mathbb R^3$. By arguments already explicit above, i.e. applying Lemma~\ref{Lemma:complexificationfirstintegral} and the extension result in \cite{mattei-moussu} (page 471), this implies that $\fa$ admits a real analytic first integral in a neighborhood of the origin $0\in \mathbb R^4$. Proceeding by induction we may then conclude that the theorem holds for the case $r=n-1$.

Let us now consider the remaining cases.

\noindent{\bf case $3 \leq r\leq n-2$}
In order to make clear the ideas involved  we will consider the case $r=3$ and $n=5$.
The corresponding linear foliation has leaves diffeomorphic to the cylinder $S^2\times \mathbb R^2$ in $\mathbb R^5$. Moreover, the original foliation  is given by a one-form $\omega=d(x_1^2 + x_2 ^2 + x_3 ^2) + {\tilde \omega}(x_1,\ldots,x_5)$.
The procedure is pretty similar to the one adopted for the case $r=2$ and $n=3$. Indeed, we consider a hyperplane $\mathbb R^4 \cong E\subset \mathbb R^5$ in general position with respect to $\fa$, given by $x_5=a_1 x_1 + a_2 x_2 + a_3 x_3 + a_4 x_4$ for some suitable choice of $a_1,\ldots,a_4$. This restriction $\tilde \fa$ is given by a one-form $\tilde \omega= d(x_1^2 + x_2 ^2 + x_ 3 ^2) + \tilde {\tilde \omega}(x_1,\ldots,x_4)$. Then by the case $r=n-1$ we conclude that $\tilde \fa$ admits a real analytic first integral in some neighborhood of the origin $0\in \mathbb R^4$.
 By the same extension arguments recurrently used we conclude that $\fa$ admits a real analytic first integral in some neighborhood of the origin $0\in \mathbb R^5$. The general remaining case is proved in a similar way by induction. This ends the proof of (ii).

\section{Proof of Theorem~\ref{Theorem:reebtypepham}}
\begin{proof}[Proof of Theorem~\ref{Theorem:reebtypepham}]
The complexification $\fa_\mathbb C$ of $\fa$ is a germ of a holomorphic codimension one foliation
at the origin $0\in \mathbb C^n$.
This is given by the complex one-form $\omega_\mathbb C$ obtained as the complexification of the form $\omega$.
Hence we have $\fa_\mathbb C: \omega_\mathbb C=0$ for $\omega_\mathbb C=dP_\mathbb C + P_\mathbb C {\tilde \omega}_{\mathbb C}$
where ${\tilde \omega}_{\mathbb C}$ is the complexification of ${\tilde \omega}$ and $P_\mathbb C=\sum\limits_{j=1}^r z_j ^d $ is the complex Pham polynomial corresponding to $P$.

We first observe that $\omega\wedge dP = P {\tilde \omega} \wedge dP + P^2 {\tilde \omega} \wedge d {\tilde \omega}$. The same holds for the complexifications
$\omega_\mathbb C \wedge d P_\mathbb C = P_\mathbb C {\tilde \omega}_{\mathbb C} \wedge d P_\mathbb C + P_{\mathbb C} ^2{\tilde \omega}_{\mathbb C} \wedge d {\tilde \omega}_{\mathbb C}$. Hence
by classical Darboux-Jouanolou criterium  the hypersurface $(P_\mathbb C=0)\subset \mathbb C^n$ is invariant
by $\fa_\mathbb C$. More than this, the first homogeneous jet of $\omega_\mathbb C$ is $dP_\mathbb C$. Let us consider  the blow-up at the origin of $\mathbb C^n$ as the map $\sigma\colon \tilde{\mathbb C^n _0} \to \mathbb C^n$, with exceptional divisor $E=\sigma^{-1}(0)\subset \tilde {\mathbb  C^n _0}$ isomorphic to the projective space $\mathbb CP(n-1)$. The inverse image of $\fa_\mathbb C$ is the foliation $(\fa_\mathbb C)^* = \sigma^*(\fa_\mathbb C)$. Denote by $R=\sum\limits_{j=1}^n z_j \frac{\partial}{\partial z_j}$ the Euler vector field. Let us write $\omega_\mathbb C = \sum\limits_{j=\nu} ^\infty \omega_{j}$ in a series of degree $j\geq \nu$ homogeneous one-forms with $\omega_\nu \ne 0$.    We shall say that $\fa_\mathbb C$ is {\it non-dicritical} if $P_{\nu+1}:=\omega_\nu(R)$ is  non-identically zero in which case it is a homogeneous polynomial of degree $\nu+1$. If this is case then the exceptional divisor $E$ is invariant by $(\fa_\mathbb C)^*$ and the singular set $\sing((\fa_\mathbb C)^*)\cap E$ is called {\it tangent cone} of $\fa_\mathbb C$ denoted by $C(\fa_\mathbb C)$. In the non-dicritical case the tangent cone is the projective hypersurface $(P_{\nu+1}=0)\subset E\simeq \mathbb CP(n-1)$.

We claim:

\begin{Claim}
$\fa_{\mathbb C}$
 is non-dicritical and has an irreducible tangent cone.
\end{Claim}
\begin{proof}[proof of the claim]
In our case we have $\omega_\mathbb C = dP_\mathbb C + P_\mathbb C {\tilde \omega}_{\mathbb C}$. Since $P_\mathbb C$ is
homogeneous, we conclude that the first homogeneous jet of $\omega_\mathbb C$ is $\omega_{\nu}=dP_\mathbb C$ and
$\nu= d -1$. Hence $\omega_\nu (R) = dP_\mathbb C (R) = d\cdot P_\mathbb C = (\nu+1) P_\mathbb C$ using the classical Euler identity for homogeneous polynomials. Hence, using the above notation we have $P_{\nu+1} = (\nu+1) P_\mathbb C$ which is not identically zero. This shows that $\fa_\mathbb C$ is non-dicritical and moreover that its tangent cone is the projective hypersurface $(P_\mathbb C =0)\subset \mathbb CP(n-1)$. Since $P_\mathbb C$ is the complex Pham polynomial in variables
$(z_1,\ldots,z_r)$ and $r \geq 3$,  which is well-known to be irreducible, we conclude that the tangent cone of $\fa_\mathbb C$ is irreducible.
\end{proof}

We can now apply the main result in \cite{Cerveau-Loray}, i.e., since the degree of the tangent cone is $\nu+1=p^s$ for some prime $p$, we conclude that $\fa_\mathbb C$ admits a holomorphic first integral in some neighborhood of the origin of $\mathbb C^n$. This implies that the foliation $\fa$ admits an analytic first integral in some neighborhood of the origin $0\in \mathbb R^n$
(Lemma~\ref{Lemma:complexificationfirstintegral}).
\end{proof}

\bibliographystyle{amsalpha}

\end{document}